\documentclass[11pt]{article} 
\usepackage{multicol} 
\usepackage{newlfont}
\usepackage{amsmath,amssymb}
\usepackage{tikz}

\def\rit{\mathbb{R}}
\def\zit{\mathbb{Z}}   
\def\nit{\mathbb{N}}

\def\qit{\mathbb{Q}} 
\def\cit{\mathbb{C}}

\newcommand{\pf}{{\em Proof.~}}
\newcommand{\qed}{\hfill~~\mbox{$\Box$}}

\newenvironment{proof}{\smallskip \noindent \pf}{\qed \bigskip}

\newtheorem{theorem}{Theorem}[section]
\newtheorem{proposition}[theorem]{Proposition}
\newtheorem{definition}[theorem]{Definition}

\newtheorem{corollary}[theorem]{Corollary}

\newtheorem{remark}[theorem]{Remark}
\newtheorem{example}[theorem]{Example}

\DeclareMathOperator{\card}{Card}
\DeclareMathOperator{\Inter}{Int}

\DeclareMathOperator{\vol}{vol}

\DeclareMathOperator{\Hom}{Hom}
\DeclareMathOperator{\gr}{gr}
\DeclareMathOperator{\Spec}{Spec}
\DeclareMathOperator{\conv}{conv}
\DeclareMathOperator{\Ehr}{Ehr}

\DeclareMathOperator{\Card}{Card}

\DeclareMathOperator{\wt}{wt}
\DeclareMathOperator{\supp}{supp}

\begin{document}


\title{\bf Ehrhart polynomials of polytopes and spectrum at infinity of Laurent polynomials}
\author{\sc Antoine Douai \\
Universit\'e C\^ote d'Azur, CNRS, LJAD, FRANCE\\
Email address: antoine.douai@univ-cotedazur.fr}

\maketitle

\begin{abstract}
Gathering different results from singularity theory, geometry and combinatorics, we show that 
the spectrum at infinity of a tame Laurent polynomial counts lattice points in polytopes and we deduce 
an effective algorithm in order to compute the Ehrhart polynomial of a simplex containing the origin as an interior point. 
\end{abstract}

\section{Introduction}

The spectrum at infinity of a tame Laurent polynomial $f$, defined by C. Sabbah in \cite{Sab}, is a sequence  $(\beta_1 , \cdots , \beta_{\mu})$ of rational numbers which is related to various concepts in singularity theory (monodromy, Hodge and Kashiwara-Malgrange filtrations, Brieskorn lattices...) and is in general difficult to handle. However, 
if $f$ is convenient and nondegenerate with respect to its Newton polytope in the sense of Kouchnirenko \cite{K} (this is generically the case), this spectrum has a very concrete description: the sum $\sum_{i=1}^{\mu} z^{\beta_i}$ is equal to the Hilbert-Poincar\'e series
of the Jacobian ring of $f$ graded by the Newton filtration, see \cite{DoSa1}. 
It follows from \cite{K} that 
$\sum_{i=1}^{\mu} z^{\beta_i}= (1-z)^n \sum_{v\in N}z^{\nu (v)}$ where $\nu$ is the Newton filtration of the Newton polytope $P$ of $f$ and $\mu$ is the normalized volume of $P$, see \cite{D}, \cite{D0}.
Because the right hand side depends only on $P$, we will also call it the {\em Newton spectrum of} $P$ and we will denote it by $\Spec_P (z)$. This is recorded in Section \ref{sec:Spectra}, where we also recall the basic definitions.

Once we have this description, it follows from the work of A. Stapledon \cite{Stapledon} that the spectrum at infinity of a convenient and nondegenerate Laurent polynomial 
(equivalently, the Newton spectrum of its Newton polytope) counts "weighted" lattice points in its Newton polytope $P$ and its integer dilates, more precisely that it is a weighted $\delta$-vector as defined in {\em loc. cit.} In particular, according to an observation made in \cite{Stapledon}, the $\delta$-vector $\delta_P (z) =\delta_0 +\delta_1 z +\cdots +\delta_n z^n$ of a lattice polytope $P$ in $\rit^n$ containing the origin as an interior point, hence its Ehrhart polynomial, can be obtained effortlessly from its Newton spectrum $\sum_{i=1}^{\mu} z^{\beta_i}$: the coefficient $\delta_i$ is equal to the number of $\beta_k$'s such that $\beta_k \in ]i-1,i]$. See Section \ref{sec:EhrAndSpectrum}, where we also recall the definitions of $\delta$-vectors and Ehrhart polynomials. It should be emphasized that the Newton spectrum of a polytope $P$ is somewhat finer than its 
$\delta$-vector: both coincide if and only if $P$ is reflexive, see Proposition \ref{prop:ReflexivevsIntegral}.

This is particularly fruitful when $P$ is a reduced simplex in $\rit^n$ in the sense of \cite{Conrads} (see Section \ref{sec:EPP} for the definition of a reduced simplex and its weight) because we have a very simple closed formula for its Newton spectrum, see Theorem \ref{theo:specEPP}. This formula only involves the weight of the simplex and is based on \cite[Theorem 1]{DoSa2} whose proof uses classical tools in singularity theory. This provides an effective algorithm in order to compute the $\delta$-vectors and the Ehrhart polynomials of reduced simplices, see Section \ref{sec:Algo}. For instance, let us consider the simplex
$$\Delta :=\conv ( (1,0, 0,0), (0,1,0,0), (0,0,1,0), (0,0,0,1),  (-1,-1,-1,-5))$$
 in $\rit^4$. 
Using Theorem \ref{theo:specEPP}, it is readily checked that its Newton spectrum is
$$\Spec_{\Delta}  (z) =1+z+z^2 +z^3 +z^4 +z^{16/5}+z^{12/5}+z^{8/5}+z^{4/5}$$
and it follows from Proposition \ref{prop:SpecEgalDelta} that $\delta_0 =1$, $\delta_1 = 2$, $\delta_2 =2$, $\delta_3 =2$, $\delta_4 =2$. Finally, its Ehrhart polynomial is
$$L_{\Delta}  (z)= \frac{1}{24}(9z^4 +10 z^3 +75 z^2 +50z+24).$$
See Example \ref{ex:ExSpecDim4Suite}. 

Notice that it is also possible to get in a similar way the $\delta$-vectors of Newton polytopes of convenient and nondegenerate polynomials from their toric Newton spectrum (as defined in \cite[Definition 3.1]{D}). This can be applied for instance to the Mordell-Pommersheim tetrahedron, that is the convex hull of $(0,0,0)$, $(a,0,0)$, $(0,b,0)$, $(0,0,c))$
where $a$, $b$ and $c$ are positive integers, which is the Newton polytope of the Brieskorn-Pham polynomial $f(u_1 ,u_2 ,u_3 )=u_1^a +u_2^b +u_3 ^c$ on $\cit ^3$, a very familiar object in singularity theory. See Section \ref{sec:ToricSpectrum}.

In this context, other questions may arise: for instance it is expected that the variance of the spectrum at infinity of a Laurent polynomial in $(\cit^* )^n$ is bounded below by $n/12$ (this is a global variant of C. Hertling's conjecture \cite{Her}, see \cite{D0} and the references therein). 
By Proposition \ref{prop:ReflexivevsIntegral}, this would give informations about the distribution of the $\delta$-vector in the reflexive case. See Remark \ref{rem:Variance}. 
In the opposite direction, the theory of polytopes is useful in order to get results on the singularity side: for instance, using \cite{Payne} it is readily seen that the spectrum at infinity of a tame regular function is not always unimodal (this is discussed in Remark \ref{rem:Unimodality}). \\

\noindent {\bf Acknowledgements.} I thank R. Davis for pointing to me the references \cite{BDS} and \cite{Payne} about unimodality.

\section{The Newton spectrum of a polytope}

\label{sec:Spectra}
 
In this section, we first recall some basic facts about polytopes that we will use. We then define the Newton spectrum of a polytope in Section \ref{sec:DefSpectrumPolytope} and we give a closed formula for the Newton spectrum of simplices in Section \ref{sec:EPP}.

\subsection{Polytopes (basics)}

\label{sec:Polytopes}

In this text, $N$ is the lattice $\zit^n$, $M$ is the dual lattice 
$\langle \ ,\ \rangle$ is the canonical pairing between $N_{\rit}$ and $M_{\rit}$.
A {\em lattice polytope} is the convex hull of a finite set of $N$.
If $P\subset N_{\rit}$ is a $n$-dimensional lattice polytope containing the origin as an interior point
there exists, for each facet $F$ of $P$, $u_F \in M_{\qit}$ such that
$$P\subset \{n\in N_{\rit},\ \langle u_F ,n\rangle \leq 1 \}\ \mbox{and}\
F=P\cap \{x\in N_{\rit},\ \langle u_F , x\rangle = 1 \}.$$
This provides the hyperplane presentation      
\begin{equation}\label{eq:PresentationPolytope}
P=\cap_F \{x\in N_{\rit},\ \langle u_F , x\rangle \leq 1 \}.
\end{equation}
 The set
\begin{equation}\label{eq:DefPolar}
P^{\circ}:=\{y\in M_{\rit},\ \langle y, x \rangle\leq 1\ \mbox{for all}\ x\in P\}
\end{equation}
is the {\em polar polytope} of $P$. It happens that $P^{\circ}$ is indeed a rational polytope (the convex hull of a finite set of $N_{\qit}$) if $P$ contains the origin as an interior point:
the vertices of $P^{\circ}$ are in one-to-one correspondance with the facets of $P$ by                   
\begin{equation}\label{eq:VerticesPolaire}
u_{F}\ \mbox{vertice of}\ P^{\circ} \Leftrightarrow \  F =\{x\in N_{\rit}, \langle u_{F}, x\rangle =1\}.
\end{equation} 
A lattice polytope $P$ is {\em reflexive} if it contains the origin as an interior point and
if $P^{\circ}$ is a lattice polytope.

If $P\subset N_{\rit}$ is a full dimensional lattice polytope containing the origin as an interior point, 
we get a complete fan $\Sigma_P$ in $N_{\rit}$ by taking the cones over the faces of $P$. We will denote by $X_{\Sigma_P}$ the complete toric variety  associated with the fan $\Sigma_P$. 
We will always assume that the fan $\Sigma_P$ is simplicial.
The {\em Newton function} of $P$ is the function
\begin{equation}\nonumber
\nu :N_{\rit}\rightarrow \rit
\end{equation}
which takes the value $1$ at the vertices of $P$ and which is linear on each cone of the fan $\Sigma_P$. 
Alternatively, $\nu (v)=\max_F <u_F , v>$ where the maximum is taken over the facets of $P$ and the vectors $u_F$ are defined as in (\ref{eq:PresentationPolytope}). The {\em Milnor number} of $P$ is 
\begin{equation}\nonumber
\mu_P :=n! \vol (P)
\end{equation}
where the volume $\vol (P)$ is normalized such that the volume of the cube is equal to $1$. If $P=\conv (b_1 ,\cdots ,b_r)$, the Milnor number $\mu_P$ is also the global Milnor number of the Laurent polynomial $f(\underline{u})=\sum_{i=1}^r c_i \underline{u}^{b_i}$ for generic complex coefficients $c_i$, that is 
$$\mu_P =\dim_{\cit}\frac{\cit [u_1 , u_1^{-1},\cdots , u_n ,u_n^{-1}]}{(u_1\frac{\partial f}{\partial u_1},\cdots , u_n\frac{\partial f}{\partial u_n})}$$ where $(\cit^* )^n$ is equipped with the coordinates 
$\underline{u}=(u_1 ,\cdots ,u_n )$, $\underline{u}^m :=u_1^{m_1}\cdots u_n^{m_n}$ if $m=(m_1 ,\cdots ,m_n )\in\zit^n$ and $(u_1\frac{\partial f}{\partial u_1},\cdots , u_n\frac{\partial f}{\partial u_n})$ denotes the ideal generated by the partial derivative $u_1\frac{\partial f}{\partial u_1},\cdots , u_n\frac{\partial f}{\partial u_n}$.
See \cite{K} for details.

\subsection{The Newton spectrum of a polytope}
\label{sec:DefSpectrumPolytope}

The main object of this paper is given by Definition \ref{def:SpecPolytope} which is motivated by the description
of the spectrum at infinity of a Laurent polynomial in Kouchnirenko's framework,
see \cite {D}, \cite{D0}, \cite{DoSa1}, \cite{Sab}. 
We briefly recall the construction and we gather the results that we will use.

Let $f(\underline{u})=\sum_{m\in\zit^n} a_m \underline{u}^{m}$ be a Laurent polynomial defined on $(\cit^* )^n$. 
The {\em Newton polytope} $P$ of $f$ is the convex hull of $\supp f :=\{m\in\zit^n , \ a_m \neq 0\}$ in $\rit^n$. We assume that $P$ contains the origin as an interior point, so that $f$ is convenient in the sense of \cite{K}. The Newton function of $P$ defines a filtration on $\mathcal{A}:=\cit [u_1 , u_1^{-1},\cdots , u_n ,u_n^{-1}]$ and, by projection, a filtration $\mathcal{N}_{\bullet}$
on $\mathcal{A}/ ( u_1\frac{\partial f}{\partial u_1},\cdots , u_n\frac{\partial f}{\partial u_n})$, see \cite[Section2]{D}, \cite[Section 4]{DoSa1}.
The {\em spectrum at infinity of } $f$ is a priori a sequence $(\beta_1 ,\cdots , \beta_{\mu_P})$ of nonnegative rational numbers (there may be repeated items), see \cite{Sab}, \cite[Section 2.e]{DoSa1}. We will denote $\Spec_f (z):=\sum_{i=1}^{\mu_P}z^{\beta_i}$.
We first quote the following result (for our purpose, one can take equality (\ref{eq:SpectrumInfinity}) as the definition of the spectrum at infinity of the Laurent polynomial $f$) which holds if $f$ is moreover nondegenerate with respect to $P$ in the sense of \cite[D\'efinition 1.19]{K}:

\begin{theorem}\cite{DoSa1}\label{theo:SpectrumInfinity}
Let $f$ be a convenient and nondegenerate Laurent polynomial defined on $(\cit^* )^n$. Then
\begin{equation}\label{eq:SpectrumInfinity}
\Spec_f (z)=\sum_{\alpha\in\qit}\dim_{\cit} \gr^{\mathcal{N}}_{\alpha}\frac{\mathcal{A}}{(u_1\frac{\partial f}{\partial u_1},\cdots , u_n\frac{\partial f}{\partial u_n})} z^{\alpha}.
\end{equation}
\end{theorem}
\begin{proof}
By \cite[Theorem 4.5]{DoSa1}, the Newton filtration and the Kashiwara-Malgrange filtration on the Brieskorn lattice of $f$ coincide if $f$ is convenient and nondegenerate. The result then follows from the description of the spectrum given in \cite[Section 2.e]{DoSa1}.
\end{proof}

\begin{corollary} \cite{D}, \cite{D0}\label{coro:SpectrumInfinitySpecPolytope}
Let $f$ be a convenient and nondegenerate Laurent polynomial defined on $(\cit^* )^n$. Then
\begin{equation}\label{eq:SpectreNewton}
\Spec_f (z)= (1-z)^n \sum_{v\in N}z^{\nu (v)}
\end{equation}
where $\nu$ is the Newton function of the Newton polytope $P$ of $f$.
\end{corollary}
\begin{proof}
Follows from Theorem \ref{theo:SpectrumInfinity} and \cite[Th\'eor\`eme 2.8]{K}, as in \cite[Theorem 3.2]{D}, 
\end{proof}

The right hand side of (\ref{eq:SpectreNewton}) depends only on $P$. This justifies the following definition:

\begin{definition}\label{def:SpecPolytope}
Let $P$ be a full dimensional lattice polytope in $N_{\rit}$, containing the origin as an interior point, and let $\nu$ be its Newton function.
The Newton spectrum of $P$ is 
\begin{equation}\label{eq:DefSpecP}
\Spec_{P}(z):=(1-z)^n \sum_{v\in N} z^{\nu (v)}.
\end{equation}
\end{definition}

\noindent Notice that various versions of the right hand of (\ref{eq:DefSpecP}) appear in many places, see for instance
\cite[Theorem 4.3]{Batyrev} for a relation with stringy $E$-functions.\\

 We have the following geometric description of the Newton spectrum of a polytope $P$, which allows to identify the spectrum at infinity of a convenient and nondegenerate polynomial whose Newton polytope is $P$. Define, for $\sigma$ a cone in the fan $\Sigma_P$ generated by the vertices $b_1 ,\cdots ,b_r$ of $P$,
\begin{equation}\nonumber
\Box (\sigma ):=\{\sum_{i=1}^{r} q_i b_i,\ q_i \in [0,1[,\ i=1,\cdots ,r\}
\end{equation}
and $\Box (\Sigma_P):=\cup_{\sigma\in\Sigma_P}\Box (\sigma )$.

\begin{proposition}\label{prop:DescrSpecP} \cite{D}, \cite{D0}, \cite{Stapledon}
Let $P$ be a full dimensional lattice polytope in $N_{\rit}$, containing the origin as an interior point. Then
\begin{equation}\label{eq:DescrSpecP}
\Spec_{P}(z)=\sum_{v\in \Box (\Sigma_P)\cap N}[\sum_{i=0}^{n-\dim \sigma (v)}\dim H^{2i}(X (\Sigma_P / \sigma (v)), \qit)z^i ]z^{\nu (v)}
\end{equation}
where $X (\Sigma_P / \sigma (v))$ is the toric variety associated with the quotient fan 
$\Sigma_P / \sigma (v)$ and $\sigma (v)$ is the smallest cone containing $v$.
\end{proposition}
\begin{proof}
Follows from \cite[Proposition 2.6]{K}, as in \cite[Proposition 3.3]{D}. See also \cite[Theorem 4.3]{Stapledon} for another proof in a slightly different context.
\end{proof}

\begin{remark}\label{rem:ConventionSpectre}
We will write $\Spec_P (z):=\sum_{i=1}^{\mu_P} z^{\beta_i}$ where $(\beta_1 ,\cdots \beta_{\mu_P})$ is a sequence of nonnegative rational numbers (counted with multiplicities) and we will freely identify the Newton spectrum of $P$ with the sequence $(\beta_1 ,\cdots \beta_{\mu_P})$. 
\end{remark}

We will make a repeated use of the following properties. In what follows, $\Inter P$ denotes the interior of $P$, $\partial P $ denotes its boundary and, for  $I\subset\rit$, $\Spec_P^{I} (z):=\sum_{\beta_i \in I} z^{\beta_i}$.

\begin{proposition}\cite{D}, \cite{D0}, \cite{DoSa1} 
\label{prop:PropertiesSpec}
Let $P$ be a full dimensional polytope in $N_{\rit}$  containing the origin in its interior and let $\nu$ be its Newton function. Then,
\begin{enumerate}
\item $\lim_{z\rightarrow 1} \Spec_P (z) =\mu_P$,
\item $\Spec_{P}^{[0,1[} (z)=\sum_{v\in \Inter P \cap N}z^{\nu (v)}$,
\item the multiplicity of $1$ in $\Spec_{P}(z)$ is $\card (\partial P\cap N)-n$,
\item $z^n \Spec_P (z^{-1})=\Spec_P (z)$.
\end{enumerate}
\end{proposition}
\noindent See \cite[Corollary 3.5]{D}. Item 1, 2 and 4 are already contained in \cite{DoSa1} and are based on Kouchnirenko's results \cite{K}. The last property means that $\Spec_{P}$ is symmetric about $\frac{n}{2}$ and is a classical specification of a singularity spectrum. In a different setting, these statements can also be found in \cite{Stapledon}.

\section{The Newton spectrum of a reduced simplex}

\label{sec:EPP}

We keep the terminology of \cite{Conrads} and we will rather denote a simplex by $\Delta$ instead of $P$.
We will say that the polytope $\Delta :=\conv (v_0 ,\cdots ,v_n )$ is an {\em integral $\zit^n$-simplex} if its vertices $v_i$ belong to the lattice $\zit^n$ and if it contains the origin as an interior point. 

\begin{definition}
Let $\Delta :=\conv (v_0 ,\cdots ,v_n )$ be an integral $\zit^n$-simplex.
The {\em weight} of $\Delta$ is
the tuple $Q(\Delta )=(q_0 ,\cdots , q_n )$ where 
\begin{equation}\label{eq:Defqi}
q_i := |\det (v_0 ,\cdots , \widehat{v_i},\cdots , v_n )|
\end{equation}
for $i=0,\cdots ,n$. 
The simplex $\Delta$ is said to be {\em reduced} if $\gcd (q_0 ,\cdots ,q_n )=1$.
\end{definition}

\begin{proposition}\label{prop:SimplexGC}
Let $\Delta :=\conv (v_0 ,\cdots ,v_n )$ be an integral $\zit^n$-simplex and let $Q(\Delta )=(q_0 ,\cdots , q_n )$ be its weight. Then,
\begin{enumerate}
\item $\mu_{\Delta}=\sum_{i=0}^n  q_{i}$ where $\mu_{\Delta}$ is the Milnor number of $\Delta$,
\item $\sum_{i=0}^n q_i v_i =0$,
\item $(v_0 ,\cdots ,v_n )$ generate $N$ if and only if $\Delta$ is reduced.
\end{enumerate}
\end{proposition}
\begin{proof} The first point follows from the definition of the $q_i$'s and
the second one follows from Cramer's rule. For the third one notice first that, because $q_i >0$, the submodule generated by $(v_0 ,\cdots , \widehat{v_i},\cdots , v_n )$ is free of rank $n$, hence the module $N_{\Delta}$ generated by $(v_0 ,\cdots , v_n )$ is free of rank $n$. By \cite[Lemma 2.4]{Conrads}
we have $\det N_{\Delta}=\gcd (q_0 ,\cdots , q_n )$ and this is the index of $N_{\Delta}$ in $N$ so that $N_{\Delta} =N$ if and only if $\det N_{\Delta}=1$.
\end{proof}

Let $\Delta$ be a $\zit^n$-integral simplex and let $Q(\Delta )=(q_0 ,\cdots ,q_n )$ be its weight. 
Let   
$$F:=\left\{\frac{\ell}{q_{i}}|\, 0\leq\ell\leq q_{i}-1,\ 0\leq i\leq n\right\}$$
and let $f_{1},\cdots , f_{k}$ be the elements of $F$ arranged by increasing order.
Define      
\begin{align}\nonumber
  S_{f_i}:=\{j|\ q_{j}f_i \in\zit\}\subset \{0,\cdots ,n\}\ \mbox{and}\
d_{i}:=\card S_{f_{i}}
\end{align}
and let $c_{0},c_{1},\cdots , c_{\mu_{\Delta} -1}$ be the sequence
$$\underbrace{f_{1},\cdots ,f_{1}}_{d_{1}},\underbrace{f_{2},\cdots ,f_{2}}_{d_{2}},\cdots ,\underbrace{f_{k},\cdots ,f_{k}}_{d_{k}}.$$

\begin{theorem} \label{theo:specEPP}  
Let $\Delta$ be a $\zit^n$-integral simplex and let $Q(\Delta )=(q_0 ,\cdots ,q_n )$ be its weight. Assume that $\Delta$ is reduced. 
Then the Newton spectrum of $\Delta$ is 
$$\Spec_{\Delta}(z)=z^{\alpha_0}+z^{\alpha_1}+\cdots +z^{\alpha_{\mu_{\Delta} -1}}$$
where 
$$\alpha_{k}:=k-\mu_{\Delta}  c_{k}$$
for $k=0,\cdots ,\mu_{\Delta}-1$. 
\end{theorem}
\begin{proof}
The vertices $v_0 ,\cdots , v_n$ of $\Delta$ satisfy $\sum_{i=0}^n q_i v_i =0$ and they generate $\zit^n$ by 
Proposition \ref{prop:SimplexGC} because $\gcd (q_0 ,\cdots , q_n )=1$. Thus, they define the exact sequence
$$0\longrightarrow \zit \longrightarrow \zit^{n+1} \longrightarrow \zit^n \longrightarrow 0$$
where the map on the right is defined by $\psi (e_i) =v_i$ for $i=0,\cdots ,n$ (we denote here by $(e_0,\cdots , e_n)$ the canonical basis of $\zit^{n+1}$) and the map on the left is defined by $\phi (1)=(q_0 ,\cdots ,q_n )$. It follows that the simplex $\Delta$ is the Newton polytope of the convenient and nondegenerate Laurent polynomial considered in \cite{DoSa2}. Thus, the assertion follows from the definition of the Newton spectrum of a polytope given in Section \ref{sec:DefSpectrumPolytope} and \cite[Theorem 1]{DoSa2}.   
\end{proof}

\begin{remark} The $\alpha_k$'s defined in Theorem \ref{theo:specEPP}   are not necessarily distinct. 
For instance, let us consider $\Delta =\conv ( (1,0), (0,1), (-1,-2))$
in $\rit^2$. We have $Q(\Delta )= (1,1,2)$, $\mu_{\Delta} =4$ and the sequence $\alpha_0 ,\alpha_1 ,\alpha_2 , \alpha_{3}$ is $0,1,2,1$. 
\end{remark}

\begin{example}\label{ex:ExSpecDim4}               
Let us consider 
$$\Delta =\conv ( (1,0, 0,0), (0,1,0,0), (0,0,1,0), (0,0,0,1),  (-1,-1,-1,-5))$$
in $\rit^4$. We have $Q(\Delta )= (1,1,1,1,5)$ and $\mu_{\Delta} =9$.
The sequence $c_{0},c_{1},\cdots , c_{8}$ is 
$$0,0,0,0,0, \frac{1}{5}, \frac{2}{5}, \frac{3}{5}, \frac{4}{5}$$
 and the sequence $\alpha_0 ,\alpha_1 ,\cdots , \alpha_{8}$ is 
$$0,1,2,3,4, 5-\frac{9}{5}, 6-\frac{18}{5}, 7- \frac{27}{5}, 8- \frac{36}{5}.$$
Theorem \ref{theo:specEPP} provides $\Spec_{\Delta} (z) =1+z+z^2 +z^3 +z^4 +z^{16/5}+z^{12/5}+z^{8/5}+z^{4/5}$.
\end{example}

\begin{remark} Theorem \ref{theo:specEPP} is not true if we do not assume that $\Delta$ is reduced:        
for instance, let $\Delta =\conv ( (2,0), (0,2), (2,2))$
in $\rit^2$. We have $Q(\Delta )= (4,4,4)$ and $\Delta$ is not reduced. By Proposition \ref{prop:PropertiesSpec} we have 
$$\Spec_{\Delta} (z) =1+3z^{1/2}+4z +3z^{3/2}+z^{2}$$
but Theorem \ref{theo:specEPP} would give $4+4z+4z^2$.
\end{remark}

\section{The spectrum as a weighted $\delta$-vector and application to the computation of (weighted) Ehrhart polynomials}

In the first part of this section, we show that the Newton spectrum of a polytope counts weighted lattice points and in the second part we write down an algorithm in order to compute Ehrhart polynomials of reduced simplices.

\subsection{The Newton spectrum of a polytope as a weighted $\delta$-vector}

\label{sec:EhrAndSpectrum}

Once the spectrum at infinity of a Laurent polynomial is identified by the formula (\ref{eq:SpectreNewton}), the results of this subsection can be basically found in \cite{Stapledon}. For the convenience of the reader we give self-contained proofs. For the background about Ehrhart theory we refer to the book \cite{BeckRobbins}.

Let $P$ be a full dimensional lattice polytope in $\rit^n$ and define, for a nonnegative integer $\ell$, $L_P (\ell ):= \card ( (\ell P )\cap M)$. Then $L_P$ is a polynomial in $\ell$ of degree $n$ (this is the {\em Ehrhart polynomial}) and we have
\begin{equation}\label{eq:serie Ehrhart}
\Ehr_P (z):=1+\sum_{m\geq 1}L_P (m) z^m =\frac{\delta_0 +\delta_1 z +\cdots +\delta_n z^n}{(1-z)^{n+1}}
\end{equation}
where the $\delta_j$'s are nonnegative integers. We will call $\Ehr_P (z)$ the {\em Ehrhart series} of $P$. 
We will write $\delta_P (z) :=\delta_0 +\delta_1 z +\cdots +\delta_n z^n$. With a slight abuse of terminology, $\delta_P (z)$ is called the {\em $\delta$-vector} of the polytope $P$. 
The Ehrhart polynomial of a polytope is extract from its $\delta$-vector by the means of the formula
\begin{equation}\label{eq:EhrBinom}
L_P (z)=\delta_0\binom{z+n}{n}+\delta_1\binom {z+n-1}{n} +\cdots +\delta_{n-1}
\binom {z+1}{n}+\delta_n \binom {z}{n}.
\end{equation}

Following \cite{Stapledon}, one defines a weighted version of the $\delta$-vectors. Let $P$ be  full dimensional lattice polytope in $\rit^n$, containing the origin as an interior point, and let $\nu$ be its Newton function.
The {\em weight} of $v\in N$ is $\wt (v):= \nu (v) -\lceil \nu (v)\rceil $ where $\lceil \ \rceil$ denotes the ceiling function.
For $m\in \nit$, let $L^{\alpha}_P (m)$ be the number of lattice points in $mP$ of weight $\alpha$ and define 
$$\delta_P^{\alpha} (z):=(1-z)^{n+1} \sum_{m\geq 0}L^{\alpha}_P (m) z^m .$$
By the very definition, we have $\delta_P (z)=\sum_{\alpha\in ]-1,0]} \delta_P^{\alpha} (z)$.

\begin{definition} The {\em weighted $\delta$-vector} of the polytope $P$ is $\delta_P^{\wt} (z):=\sum_{\alpha\in ]-1,0]} \delta_P^{\alpha} (z) z^{\alpha}$.
\end{definition}

\begin{theorem} \label{theo:SpecvsDeltatwist}
Let $P$ be a full dimensional lattice polytope containing the origin as a interior point. Then,
\begin{equation}\nonumber
\Spec_P (z)=\delta_P^{\wt} (z).
\end{equation}
In particular, the spectrum at infinity $\Spec_f (z)$ of a convenient and nondegenerate Laurent polynomial $f$ is 
equal to the weighted $\delta$-vector $\delta_P^{\wt} (z)$.
\end{theorem}
\begin{proof}
Because $v\in mP\cap N$ if and only if $\nu (v)\leq m$, we get
$$\delta_P^{\wt} (z)=(1-z)^{n+1} \sum_{m\geq 0}(\sum_{v\in mP}z^{wt (v)})z^m =
(1-z)^{n+1} \sum_{m\geq 0}(\sum_{\nu (v)\leq m}z^{\nu (v)-\lceil v\rceil})z^m $$
$$
=(1-z)^{n+1} \sum_{m\geq 0}(\sum_{\lceil \nu (v)\rceil\leq m}z^{\nu (v)-\lceil v\rceil})z^m 
=
(1-z)^{n+1} \sum_{v\in N}(\sum_{\lceil \nu (v)\rceil\leq m}z^{m-\lceil v\rceil})z^{\nu (v)} $$
$$=(1-z)^{n} \sum_{v\in N}z^{\nu (v)}=\Spec_P (z).$$ 
The last assertion follows from (\ref{eq:SpectreNewton}) and (\ref{eq:DefSpecP}).
\end{proof}

\noindent To sum up, the spectrum at infinity of a Laurent polynomial counts weighted lattice points in its Newton polytope.

\begin{corollary} \label{coro:DescrDeltaAlpha} 
Let $\alpha \in ]-1,0]$ such that $\delta^{\alpha}_P (z)$ is not identically equal to zero. Then, with the notations of Proposition \ref{prop:DescrSpecP}, the rational number
$\alpha$ is equal to $\nu (v)-\lceil \nu (v)\rceil$ for a suitable $v\in \Box (\Sigma_P)\cap N$ and
\begin{equation}\label{eq:DescrDeltaAlpha}
\delta^{\alpha}_P (z)=c_v^{0} z^{\lceil \nu (v)\rceil}+\cdots +c_v^{n-\dim \sigma (v)} z^{\lceil \nu (v)\rceil+n-\dim \sigma (v)}
\end{equation}
where $c_v^i :=\dim H^{2i}(X (\Delta/ \sigma (v)), \qit)$.
\end{corollary}
\begin{proof} Follows from Proposition \ref{prop:DescrSpecP} and Theorem \ref{theo:SpecvsDeltatwist}.
\end{proof}

\begin{corollary} \label{coro:deltaatmostn} 
Let $\alpha \in ]-1,0]$. Then,
\begin{enumerate}
\item $\delta^{\alpha}_P (z)$ is a polynomial of degree at most $n$ in $z$ whose coefficients are nonnegative integers,
\item $L_P^{\alpha}$ is a polynomial of degree $n$ (if not identically equal to zero) and  
\begin{equation}\label{eq:DeltaEhrhart}
L_P^{\alpha} (z)=\delta_0^{\alpha}\binom{z+n}{n}+\delta_1^{\alpha} \binom {z+n-1}{n} +\cdots +\delta_{n-1}^{\alpha} 
\binom {z+1}{n}+\delta_n^{\alpha} \binom {z}{n}
\end{equation}
if $\delta^{\alpha}_P (z) =\delta_0^{\alpha} +\delta_1^{\alpha} z +\cdots +\delta_n^{\alpha} z^n$.
\end{enumerate}
\end{corollary}
\begin{proof} The first assertion follows from Corollary \ref{coro:DescrDeltaAlpha} because
where $c_v^i :=\dim H^{2i}(X (\Delta/ \sigma (v)), \qit)$ is a nonnegative integer and because $\lceil \nu (v)\rceil \leq \dim\sigma (v)$ (recall that the Newton spectrum is contained in $[0,n]$, see Proposition \ref{prop:PropertiesSpec}). The second one is straightforward.
\end{proof}

The next observation is borrowed from \cite{Stapledon}:

\begin{proposition}\label{prop:SpecEgalDelta}
Let $P$ be a full dimensional lattice polytope containing the origin as a interior point.
Let $\delta_P (z) :=\sum_{k=0}^n \delta_k z^k $ be the $\delta$-vector of the $P$ and let $\Spec_P (z):=\sum_{i=1}^{\mu_P} z^{\beta_i}$ be its Newton spectrum. Then, for $k=0,\cdots ,n$, the coefficient $\delta_k$ is equal to the number of $\beta_i$'s, $1\leq i\leq \mu_P$, such that $\beta_i\in ]k-1,k]$.
\end{proposition}
\begin{proof} 
The result follows from Corollary \ref{coro:DescrDeltaAlpha} because $\delta_P (z)=\sum_{\alpha\in ]-1,0]} \delta_P^{\alpha} (z)$. 
\end{proof}

\noindent Using Proposition \ref{prop:PropertiesSpec}, we check that $\delta_0 =1$, $\delta_1 =\Card (P\cap N)-(n+1)$ and $\delta_0 +\cdots +\delta_n =\mu_P$ where $\mu_P$ is the Milnor number of $P$.

\subsection{An algorithm to compute (weighted) Ehrhart polynomials and $\delta$-vectors of reduced simplices}
\label{sec:Algo}

Let $P$ be a full dimensional lattice polytope in $\rit^n$ containing the origin in its interior. 
A general recipe in order to compute the $\delta$-vector of $P$ is now clear:
compute the Newton spectrum of $P$ and use Proposition \ref{prop:SpecEgalDelta}.

This provides an effective algorithm in order to compute the (weighted) $\delta$-vectors and the (weighted) Ehrhart polynomials of a reduced simplex $\Delta$ containing the origin as an interior point because the Newton spectrum of $\Delta$ is known:
\begin{itemize}
\item compute the weight of $\Delta$ (use equation (\ref{eq:Defqi})) and check that $\Delta$ is reduced,
\item compute the Newton spectrum of $\Delta$ using Theorem \ref{theo:specEPP},
\item  use Theorem \ref{theo:SpecvsDeltatwist}, Corollary \ref{coro:deltaatmostn} and equation (\ref{eq:DescrDeltaAlpha}) in order to compute the weighted $\delta$-vectors 
$\delta_{\Delta}^{\alpha} (z)$ and weighted Ehrhart polynomials $L_{\Delta}^{\alpha} (z)$,
\item  use Proposition \ref{prop:SpecEgalDelta} in order to get the $\delta$-vector $\delta_{\Delta} (z) :=\delta_0 +\delta_1 z +\cdots +\delta_n z^n$ of $\Delta$ and formula (\ref{eq:EhrBinom})  in order to get the Ehrhart polynomial $L_{\Delta} (z)$ of $\Delta$.
\end{itemize}

\begin{example}                      
Let 
$$\Delta :=\conv ( (1,0, 0), (0,2,0), (1, 1,1), (-3,-5,-2))$$
in $\rit^3$. Its weight is $Q(\Delta )= (2,2,3,4)$ and $\Delta $ is reduced.
We have $\mu_{\Delta}  =11$ and
$$\Spec_{\Delta}  (z) =1+z+z^2 +z^3 +z^{5/4}+z^{4/3}+z^{1/2}+z^{3/2} +z^{5/2} +z^{5/3}+z^{7/4}$$
by Theorem \ref{theo:specEPP}. 
Proposition \ref{prop:SpecEgalDelta} provides $\delta_0 =1$, $\delta_1 = 2$, $\delta_2 =6$, $\delta_3 =2$ 
and we get
$$L_{\Delta}  (z)= \frac{1}{6}(11z^3 +6z^2 +13 z +6).$$
\end{example}

\begin{example} \label{ex:ExSpecDim4Suite}   (Example \ref{ex:ExSpecDim4} continued)               
Let us consider the simplex
$$\Delta :=\conv ( (1,0, 0,0), (0,1,0,0), (0,0,1,0), (0,0,0,1),  (-1,-1,-1,-5))$$
 in $\rit^4$. 
Its weight is $Q(\Delta )= (1,1,1,1,5)$ and $\Delta$ is reduced.
We have $\mu_{\Delta}  =9$ and
$$\Spec_{\Delta}  (z) =1+z+z^2 +z^3 +z^4 +z^{16/5}+z^{12/5}+z^{8/5}+z^{4/5}$$
by Example \ref{ex:ExSpecDim4}.            
Proposition \ref{prop:SpecEgalDelta} provides $\delta_0 =1$, $\delta_1 = 2$, $\delta_2 =2$, $\delta_3 =2$, $\delta_4 =2$ 
and we get
$$L_{\Delta}  (z)= \frac{1}{24}(9z^4 +10 z^3 +75 z^2 +50z+24).$$
We have also
$\delta_{\Delta}^0 (z) =1+z+z^2 +z^3 +z^4$, $\delta_{\Delta}^{-1/5} (z) =z$, $\delta_{\Delta}^{-2/5} (z) =z^2$,
$\delta_{\Delta}^{-3/5} (z) =z^3$, $\delta_{\Delta}^{-4/5} (z) =z^4$
and
$$L_{\Delta}^{-1/5} (z)=\frac{1}{24}(z^4 +6z^3 +11z^2 +6z),\ L_{\Delta}^{-2/5} (z)=\frac{1}{24}(z^4 +2z^3 -z^2-2z),$$
$$L_{\Delta}^{-3/5} (z)=\frac{1}{24}(z^4 -2z^3 -z^2 +2z),\  L_{\Delta}^{-4/5} (z)=\frac{1}{24}(z^4 -6z^3 11z^2 -6z),$$
$$L_{\Delta}^{0} (z)= \frac{1}{24}(5z^4 +10z^3 +55z^2 +50z +24).$$
In particular, weighted Ehrhart polynomials may have negative coefficients.
\end{example}

\noindent If $q_0 =1$ and if $\Delta$ is reflexive, another formula (with a different proof) for the $\delta$-vector of $\Delta$ in terms of the entries of the vector $Q(\Delta )$ is given in \cite[Theorem 2.2]{BDS}.

\subsection{Ehrhart polynomials of Newton polytopes of polynomials}
\label{sec:ToricSpectrum}

We give a few words about Newton polytopes of {\em polynomials} (and not Laurent polynomials). We briefly recall the framework.
The {\em support} of a polynomial $g=\sum_{m\in \nit^n}a_m u^m \in\cit [u_1 ,\cdots , u_n ]$, where we write $u^m:= u_1^{m_1}\cdots u_n^{m_n}$ if $m=(m_1 ,\cdots , m_n )\in\nit^n$, is
$\supp (g)=\{m\in\nit^n,\ a_m \neq 0\}$ 
and the {\em Newton polytope} of the polynomial $g$ is the convex hull of $\{0\}\cup \supp (g)$ in $\rit^n_+$.
The {\em toric} Newton spectrum of $g$ (or the toric Newton spectrum of its Newton polytope) is 
$$\sum_{\alpha\in\qit}\dim_{\cit} \gr^{\mathcal{N}}_{\alpha}\frac{\cit [u_1 ,\cdots ,u_n ]}{(u_1\frac{\partial g}{\partial u_1},\cdots , u_n\frac{\partial g}{\partial u_n})} z^{\alpha},$$  
see \cite[Definition 3.1]{D}.
Thanks to \cite[Theorem 3.2]{D}, Theorem \ref{theo:SpecvsDeltatwist} and its corollaries (in particular Proposition \ref{prop:SpecEgalDelta}) still hold, with the same proofs, for the toric Newton spectrum of the Newton polytope of a convenient and nondegenerate polynomials. This gives a recipe in order to calculate Ehrahrt polynomials of Newton polytopes of polynomials from their toric Newton spectrum.
 We give some examples.
\begin{enumerate}
\item Let 
$$\Delta :=\conv ((0,0,0), (a,0,0), (0,b,0), (0,0,c))$$
where $a$, $b$ and $c$ are positive integers. 
 It is the Newton polytope of $f(u_1 , u_2 ,u_3 )=u_1^a +u_2^b +u_3^c$. Its toric Newton spectrum is the sequence of the following rational numbers: 
\begin{itemize}
\item $0$, 
\item $\frac{i}{a}$ for $1\leq i\leq a-1$,
$\frac{i}{b}$ for $1\leq i\leq b-1$, $\frac{i}{c}$ for $1\leq i\leq c-1$, 
\item $\frac{i}{a}+\frac{j}{b}$ for $1\leq i\leq a-1$ and $1\leq j\leq b-1$,
$\frac{i}{a}+\frac{j}{c}$ for $1\leq i\leq a-1$ and $1\leq j\leq c-1$,
$\frac{i}{b}+\frac{j}{c}$ for $1\leq i\leq b-1$ and $1\leq j\leq c-1$, 
\item $\frac{i}{a}+\frac{j}{b}+\frac{k}{c}$ for $1\leq i\leq a-1$, $1\leq j\leq b-1$ and $1\leq k\leq c-1$.
\end{itemize}
Its $\delta$-vector is then given by Proposition \ref{prop:SpecEgalDelta}.
For instance, if $a=2$, $b=3$ and $c=3$
we get $\delta_0 =1$, $\delta_1 =10$, $\delta_2 =7$, $\delta_3 =0$. If $a=1$, $b=1$ and $c=1$ (this corresponds to the standard simplex in $\rit^n$) we get $\delta_0 =1$, $\delta_1 =\delta_2 =\delta_3 =0$.
\item Let us consider the polytope in $\rit^3$
$$P:=\conv ((0,0,0),(1,0,0), (0,1,0), (0,0,1), (1,1,h))$$
where $h$ is an integer greater or equal to $2$. This is a variation of Reeve's tetrahedron, see \cite[Example 3.22]{BeckRobbins}. It is the Newton polytope of $f(u_1 , u_2 , u_3 )=u_1 +u_2 +u_3 +u_1 u_2 u_3^h$
whose toric Newton spectrum is equal to
$1+z+z^2 +\sum_{i=1}^{h-1} z^{1+\frac{i}{h}}$.
Thus, $\delta_0 =1$, $\delta_1 =1$, $\delta_2 =h$, $\delta_3 =0$ and 
$$L_{P} (z)=\frac{1}{6}[(h+2) z^3 +9z^2 +(13-h)z +6].$$ 
Notice that the coefficient of $z$ is negative if $h\geq 14$.
\end{enumerate}

\section{Reflexive polytopes and their Newton spectrum}

\label{sec:ReflAndInt}

We give in this section various characterizations of reflexive polytopes involving their Newton spectrum. 
We show in particular that the spectrum of a polytope is equal to its $\delta$-vector if and only if $P$ is reflexive.

\begin{proposition}\label{prop:ReflexivevsIntegral}
Let $P$ be a polytope containing the origin in its interior and let $\delta_P (z) =\delta_0 +\delta_1 z +\cdots +\delta_n z^n$ be its $\delta$-vector. Then the following are equivalent:
\begin{enumerate}
\item $\Spec_P (z)$ is a polynomial, 
\item $P$ is reflexive,
\item $\Spec_{P}(z)=\delta_P (z):=\delta_0 +\delta_1 z +\cdots +\delta_n z^n$.
\end{enumerate}
\end{proposition}
\begin{proof} $1 \Leftrightarrow 2$.
 Assume that $P$ is reflexive: if
$n$ belongs to the cone $\sigma_F$ supported by the facet $F$, the Newton function of $P$ is defined by 
$\nu : n\mapsto <u_F ,n >$ where $u_F \in M$ by (\ref{eq:VerticesPolaire}). By Proposition \ref{prop:DescrSpecP}, $\Spec_P (z)$ is a polynomial. 
Conversely, assume that $\Spec_P (z)$ is a polynomial. Let $F$ be a facet of $P$. Again by Proposition \ref{prop:DescrSpecP}, the Newton function $\nu$ of $P$ takes integral values on the fundamental domain
$\Box(\sigma_F )$. It follows that 
$\nu$ takes integral values on $\sigma_F\cap N$ and, by linearity, we get a $\zit$-linear map 
$\nu_{\sigma_F} : N_{\sigma_F} \rightarrow \zit$ where $N_{\sigma_F}$ is the sublattice of $N$ generated by the points of $\sigma_F \cap N$. 
Because we have the isomorphism $\Hom_{\zit}(N_{\sigma_F}, \zit )\cong M/\sigma_{F}^{\perp}\cap M$ induced by the dual pairing between $M$ and $N$, there exists $u_{F}\in M$ such that
$\nu (n)=\langle u_{F}, n \rangle$ when $n\in \sigma_F$. In particular, the equation of $F$ is
$\langle u_{F}, x \rangle=1$ for $u_F \in M$. Thus, $P$ is reflexive. 

\noindent $1\Leftrightarrow 3$.  Assume that $\Spec_P (z)$ is a polynomial: by Proposition \ref{prop:SpecEgalDelta}, 
$\delta_i =\Card \{ \alpha\in \Spec_P ,\ \alpha =i\} $ thus $\Spec_P (z)=\delta_P (z)$. Conversely, if 
$\Spec_P (z)=\delta_P (z)$ the Newton spectrum is obviously a polynomial.
\end{proof}

\begin{remark}\label{rem:Unimodality} 
Recall that a polynomial $a_0 +a_1 z+\cdots +a_n z^n$ is unimodal if there exists an index $j$ such that $a_i \leq a_{i+1}$ for all $i<j$ and $a_i \geq a_{i+1}$ for all $i\geq j$. It follows from Corollary \ref{coro:SpectrumInfinitySpecPolytope}, Proposition \ref{prop:ReflexivevsIntegral}
and \cite{MustataPayne}, \cite{Payne} that the spectrum at infinity of a Laurent polynomial is not always unimodal (and this corrects a misguided assertion made in \cite {D0}). This can be directly checked: let 
$$\Delta :=\conv (e_1 ,\cdots ,e_9 ,-\sum_{i=1}^9 q_i e_i )$$
where $(e_1 ,\cdots , e_9 )$ is the standard basis of $\rit^9$ and $(q_1 ,\cdots , q_9 ):= (1, \cdots ,1,3)$ (this example is borrowed from \cite{Payne}, see also \cite[Section 2]{BDS}). Then 
$\mu_{\Delta}=12$ and $Q(\Delta )=(1,\cdots , 1, 3)$ where $1$ is counted $9$-times: the simplex $\Delta$ is reduced and reflexive. 
We get from  Theorem \ref{theo:specEPP}
$$\Spec_{\Delta}(z)=1 + z + z^2 +2z^3 + z^4 +z^5 +2z^6 +z^7 +z^8 +z^9$$
and this polynomial is not unimodal. It follows that the spectrum at infinity of the Laurent polynomial
$$f(u_1 ,\cdots , u_9 )=u_1 +\cdots +u_9 +\frac{1}{u_1 \cdots u_8 u_9^3}$$
(see Section \ref{sec:DefSpectrumPolytope}) is not unimodal.   
\end{remark}

\begin{remark}\label{rem:Variance} 
Let $P$ be a reflexive polytope and let $\Spec_{P}(z)=\delta_P (z):=\delta_0 +\delta_1 z +\cdots +\delta_n z^n$ be its $\delta$-vector.
The mean of the $\delta$-vector is $\frac{1}{\mu_{P}}\sum_{i=0}^n i\delta_i $ (and this is equal to $\frac{n}{2}$ because $\delta_i =\delta_{n-i}$ if $P$ is reflexive, as it could be deduced from Proposition \ref{prop:ReflexivevsIntegral}) and its variance is $\frac{1}{\mu_{P}}\sum_{i=0}^n \delta_i (i-\frac{n}{2})^2$. 
After Corollary \ref{coro:SpectrumInfinitySpecPolytope}, Proposition \ref{prop:ReflexivevsIntegral}
and the version of Hertling's conjecture mentioned in the introduction (see \cite{D0} for details),
the inequality $\frac{1}{\mu_{P}}\sum_{i=0}^n \delta_i (i-\frac{n}{2})^2 \geq \frac{n}{12}$ is expected. 
\end{remark}


\begin{thebibliography}{999} 
\bibitem[1]{Batyrev} Batyrev, V.: {\em Stringy Hodge numbers of varieties with Gorenstein canonical singularities},
In: M-H Saito (ed.) et al., Integrable Systems and Algebraic Geometry, Proceedings of the 41st 
Taniguchi Symposium, Japan 1997, World Scientific, 1998, p. 1-32.
\bibitem[2]{BeckRobbins} Beck, M. , Robbins, S.: {\em Computing the continuous discretely}, Springer, New York, 2007.
\bibitem[3]{BDS} Braun, B., Davis, R., Solus, L.: {\em Detecting the integer decomposition property and Ehrhart unimodality in reflexive simplices}, Advances in Applied Math., {\bf 100}, 2018, p. 122-142.
\bibitem[4]{Conrads} Conrads, H. : {\em Weighted projective spaces and reflexive simplices}, Manuscripta Math., {\bf 107} 2002, p. 215-227.
\bibitem[5]{D} Douai, A.: {\em A note on the Newton spectrum of a polynomial}, arXiv:1810.03901.
\bibitem[6]{D0} Douai, A.: {\em Global spectra, polytopes and stacky invariants}, Math. Z.,  {\bf 288} (3), 2018,  p. 889-913. 
\bibitem[7]{DoSa1} Douai, A., Sabbah, C.: {\em Gauss-Manin systems, Brieskorn
lattices and Frobenius structures I}, Ann. Inst. Fourier, {\bf 53} (4), 2003, p. 1055-1116.
 \bibitem[8]{DoSa2} Douai, A., Sabbah, C.: {\em Gauss-Manin systems, Brieskorn lattices and Frobenius structures II}, In : Frobenius Manifolds, C. Hertling and M. Marcolli (Eds.), Aspects of Mathematics E 36 (2004).
\bibitem[9]{Her} Hertling, C.: {\em Frobenius manifolds and variance of the spectral numbers}, In : New Developpments in Singularity Theory, D. Siersma, C.T.C Wall, V. Zakalyukin (Eds.), NATO Science Series (Series II: Mathematics, Physics and Chemistry), {\bf 21}, Springer, Dordrecht.
\bibitem[10]{K} Kouchnirenko, A.G.: {\em Poly\`edres de Newton et nombres de Milnor}, Invent. Math.,  {\bf 32}, 1976, p. 1-31.


\bibitem[11]{MustataPayne} Musta\c{t}\u{a}, M., Payne, S.: {\em Ehrhart polynomials and stringy Betti numbers},
Mathematische Annalen, {\bf 333}(4), 787-795 (2005).



\bibitem[12]{Payne} Payne, S.: {\em Ehrhart series and lattice triangulations}, Discrete Comput. Geom., {\bf 40-3}, 2008, p. 365-376.
\bibitem[13]{Sab} Sabbah, C.: {\em Hypergeometric periods for a tame polynomial}, Portugalia Mathematicae, {\bf 63}, 2006, p. 173-226.
\bibitem[14]{Stapledon} Stapledon, A.: {\em Weighted Ehrhart Theory and Orbifold Cohomology}, Adv. Math., {\bf 219},  2008, p. 63-88. 
\end{thebibliography}
\end{document}